\documentclass{article}
\usepackage[utf8]{inputenc}
\usepackage[boxed]{algorithm2e}
\usepackage{listings}
\usepackage[numbers]{natbib}
\usepackage{amsmath}
\usepackage{amsthm}
\usepackage{amsfonts}
\usepackage{amssymb}
\usepackage[utf8]{inputenc}
\usepackage[english]{babel}
\usepackage[a4paper]{geometry}
\usepackage{graphicx}
\usepackage{color,psfrag}
\usepackage{fancyhdr}
\usepackage{makeidx}
\usepackage{enumerate}
\usepackage{mathpazo}
\usepackage{mathrsfs}
\usepackage{float}
\usepackage{datetime}
\usepackage{titlesec}
\usepackage{enumerate}
\usepackage{enumitem}
\usepackage[labelfont=bf,textfont=it,labelsep=period,width=.8\textwidth]{caption}
\usepackage[unicode,pdfstartview={XYZ null null 1},pdfview={XYZ null null 1},pdfpagemode=UseOutlines]{hyperref}
\usepackage[all]{xy}
\usepackage{comment}

\DeclareFontFamily{U}{wncy}{}
\DeclareFontShape{U}{wncy}{m}{n}{<->wncyr10}{}
\DeclareSymbolFont{mcy}{U}{wncy}{m}{n}
\DeclareMathSymbol{\Sh}{\mathord}{mcy}{"58}

\newcommand{\F}{\mathbb F}

\newcommand{\FF}{\mathbb{F}}

\newtheorem{theorem}{Theorem}[section]

\newtheorem{lemma}[theorem]{Lemma}

\newtheorem{corollary}[theorem]{Corollary}
\newtheorem{proposition}[theorem]{Proposition}

\newtheorem{definition}[theorem]{Definition}

\theoremstyle{remark}
\theoremstyle{remark}\newtheorem{remark}[theorem]{Remark}

\begin{document}

\title{Finding nontrivial zeros of quadratic forms over rational function fields of characteristic 2}

\author{Tímea Csahók, Péter Kutas, Mickaël Montessinos, Gergely Zábrádi\footnote{Supported by MTA R\'enyi Institute Lend\"ulet Automorphic Research Group, by the Thematic Excellence Programme, Industry and Digitization Subprogramme, NRDI Office, 2020, and by the NKFIH Research grants FK-127906 and K-135885.}}

\maketitle

\begin{abstract}
 We propose polynomial-time algorithms for finding nontrivial zeros of quadratic forms with four variables over rational function fields of characteristic 2. We apply these results to find prescribed quadratic subfields of quaternion division division algebras and zero divisors in $M_2(D)$, the full matrix algebra over a division algebra, given by structure constants. We also provide an implementation of our results in MAGMA which shows that the algorithms are truly practical. 
\end{abstract}

\section{Introduction}

The theory of quadratic spaces has a long history in mathematics and has applications in topology, number theory, algebraic geometry and in many other areas of mathematics. Two quadratic forms are equivalent if there is an invertible linear change of variables transforming one form into the other one (or alternatively, there exists an vector space isomorphism between their corresponding quadratic spaces that also respects the quadratic structure). The theory of quadratic forms is vastly different in characteristic 2 and in any other characteristic. Nevertheless, the concept of equivalence is key in both cases. A quadratic form is called isotropic if it admits a nontrivial zero and is called anisotropic otherwise. Quadratic spaces have a well-known decomposition into the direct sum of special quadratic subspaces containing isotropic vectors (called hyperbolic planes) and an anisotropic part. This somehow motivates the fact that computing isotropic vectors is useful in determining equivalence of quadratic forms. 

Let $K=\mathbb{F}_q(t)$ be the rational function field in one variable, where $q$ is an odd prime power. In \cite{ivanyos2019explicit} the authors describe a polynomial-time algorithm that decides whether two quadratic forms over $K$ are equivalent, and if so, finds an explicit equivalence between them. The key tool is a subroutine that finds isotropic vectors of the form. The algorithm doesn't naturally generalize to field extensions and doesn't work when $q$ is a power of 2. Computing isotropic vectors of quadratic forms in odd characteristic function fields (i.e., finite extensions of $\mathbb{F}_q(t)$) is considered in \cite{koprowski2021isotropic}. The algorithm works for any extensions but is not claimed to run in polynomial time. Furthermore, as demonstrated in \cite{ivanyos2019explicit}, quadratic form algorithms can be used to find zero divisors in quaternion algebras over quadratic field extensions. 

The only known quadratic form algorithm in the characteristic 2 case comes from the well-known correspondence between quaternion algebras and ternary quadratic forms. Since in \cite{ivanyos2018computing} the main algorithm can find zero divisors in quaternion algebras it can be used to find nontrivial zeros of ternary quadratic forms. In this work we consider the algorithmic problem of finding nontrivial zeros of quadratic forms over $\mathbb{F}_{2^k}(t)$ in 4 variables. Our contributions are the following:
\begin{itemize}
    \item We propose a polynomial-time algorithm that decides whether a four-variable form is isotropic or not. If it is, it also outputs a nontrivial zero.
    \item We provide a Magma implementation for finding zeros of ternary quadratic forms. Even though the algorithm is not novel it hasn't been implemented before.
    \item An implementation of our main algorithm in Magma \cite{magma}.
\end{itemize}
The paper is structured as follows. In Section \ref{sec:prelim} we recall theoretical and algorithmic preliminaries. In Section \ref{sec:main} we describe our algorithm for finding nontrivial zeros. We also provide some applications of this result, such constructing quaternion algebras with prescribed Hasse invariants, finding zero divisors in $M_2(D)$ where $D$ is a quaternion algebra over $\mathbb{F}_{2^k}(t)$ and finding prescribed maximal subfields in quaternion algebras. In Section \ref{sec:impl} we provide details about our Magma implementation of our main algorithm. For the implementation, see \url{https://github.com/Char2QuadForms/Char2QuadForms}. In Appendix \ref{appendix:algo} we give pseudo-code algorithms for the main subroutines necessary for algorithm \ref{algofinding}.

\section{Preliminaries}\label{sec:prelim} 

\subsection{Number theory background}

We end this section by stating some classical results independent of the characteristic (even though we only use them in characteristic $2$).

In this section we collect the background we need from Number theory. The following discussion is independent of the characteristic.

We are going to use the following higher dimensional variant of Hensel's lemma. Let $O$ be a complete discrete valuation ring with maximal ideal $P$. Given a multivariate polynomial $f(x_1,\dots,x_n)\in O[x_1,\dots,x_n]$ such that the gradient $(\frac{\partial f}{\partial x_1},\dots,\frac{\partial f}{\partial x_n})$ is nonzero modulo $P$ at a modulo $P$ solution $(\overline{u_1},\dots,\overline{u_n})$ then this lifts to a solution in $O$. However, the lift is not unique in general: using the one-variable Hensel's lemma one can even choose an arbitrary lift of $\overline{u_1},\dots,\overline{u_{j-1}},\overline{u_{j+1}}, \dots,\overline{u_n}$ for any $1\leq j\leq n$ with $\frac{\partial f}{\partial x_j}(u_1,\dots,u_n)\not\equiv 0\pmod{P}$.

We state a variant of the Hasse--Minkowski theorem over the field $\FF(t)$ of rational functions over a finite field $\FF$ \cite[Chapter VI, 3.1]{lam2005introduction}. It was proved by Hasse's doctoral student
Herbert Rauter in 1926 \cite{rauter1926uber}.
\begin{theorem}\label{HM}
A non-degenerate quadratic form over $\FF(t)$ is isotropic over $\FF(t)$ if and only if it is isotropic over every completion of $\FF(t)$.
\end{theorem}

For ternary quadratic forms there exists a slightly stronger version of this theorem which is a consequence of the product formula for quaternion algebras or Hilbert's reciprocity law \cite[Chapter IX, Theorem 4.6]{lam2005introduction}:
\begin{theorem}\label{Product formula}
Let $Q$ be a ternary non-degenerate quadratic form over $\FF(t)$. Then if it is isotropic in every completion except maybe one then it is isotropic over $\FF(t)$.
\end{theorem}

Finally, we need the following version (extract) of the local reciprocity law for function fields.
\begin{theorem}[Thm.\ I.1.1, Cor.\ I.1.2, Prop.\ III.1.2 in \cite{milne1997class}]\label{LCFTneeded}
Let $K$ be a nonarchimedean local field. Then the map $L\mapsto N_{L/K}(L^\times)$ is a bijection from the set of finite abelian extensions $L$ of $K$ to the norm subgroups in $K^\times$. Further, for any abelian extension $L|K$ we have $\operatorname{Gal}(L/K)\cong K^\times/N_{L/K}(L^\times)$. If $L/K$ is unramified then we have $N_{L/K}(\mathcal{O}_L^\times)=\mathcal{O}_K^\times$. Here $\mathcal{O}_K$ (resp.\ $\mathcal{O}_L$) denotes the valuation ring in $K$ (resp.\ in $L$).
\end{theorem}

\subsection{Algorithmic preliminaries}

\paragraph{Quadratic form algorithms}
Quadratic forms over fields of characteristic different from 2 have a long algorithmic history. When $char(K)\neq 2$, then the theory of ternary quadratic forms has a close connection to quaternion algebras over $\mathbb{Q}$. Namely finding a nontrivial zero of the quadratic form $ax^2+by^2-cz^2$ (where $c\neq 0$) is equivalent to finding a zero divisor in the quaternion algebra $(\frac{a}{c},\frac{b}{c})$. This algorithmic correspondance is exploited in \cite{ivanyos1996lattice} to provide an algorithm for finding zeros of indefinite rational ternary quadratic forms. In \cite{cremona2003efficient} a more direct approach is followed which is also used by \cite{van2006solving} in the case where $K$ is a rational function field. Every approach uses lattice reduction in some fashion. None of these approaches generalize to extension fields (even quadratic extensions).

In \cite{simon2005quadratic} Simon proposes an algorithm which finds nontrivial zeros of quadratic forms in four or more variables. The main idea of the algorithm is the following. Let $Q_1$ be a quadratic form in 4 variables. Then one first finds a quadratic form $Q_2$ of dimension 2 such that the orthogonal sum of the corresponding quadratic spaces results ina hyperbolic space (direct sum of hyperbolic planes). Then one can use the algorithm from \cite{simon2005solving} to compute a maximal isotropic subspace of the new quadratic space (which will have dimension 3 in this case). This will have a nontrivial intersection with the original 4-dimensional quadratic space and the intersection can be computed efficiently. Any nonzero element in the intersection corresponds to nontrivial zero. The main algorithmic tool in finding the suitable form $Q_2$ is the computation of the 2-Sylow part of a certain class group of an imaginary quadratic field. The algorithm requires an oracle for factoring the discriminant of the form which was known to be necessary for forms with 4 variables. Interestingly, under GRH, Castel \cite{castel2013solving} showed that when the number of variables is at least 5, then one can adapt Simon's algorithm in a way that a factoring oracle is no longer necessary. 
The case of finding nontrivial zeros over rational function fields with arbitrary many variables was considered in \cite{ivanyos2019explicit}. The main idea of this algorithm is quite simple: split the 4-variable form into two binary forms and find a common value they both represent. The main theoretical tool here is the local global principle and the following efficient formula for the number of monic irreducible polynomials in a given residue class of a given degree \cite{wan1997generators}:
\begin{lemma}\label{number}
Let $a,m\in\mathbb{F}_q[t]$ be such that $deg(m)>0$ and the $gcd(a,m)=1$. Let $N$ be a positive integer and let  
$$S_N(a,m)=\#\{f\in\mathbb{F}_q[t]~\text{monic irred.}~|~f\equiv a ~(mod ~m),~ deg(f)=N\}.$$
Let $M=deg(m)$ and let $\Phi(m)$ denote the number of polynomials in $\mathbb{F}_q[t]$ relative prime to $m$ whose degree is smaller than M. Then we have the following inequality:
\begin{equation*}
|S_N(a,m)-\frac{q^N}{\Phi(m)N}|\leq \frac{1}{N}(M+1)q^{\frac{N}{2}}.
\end{equation*}
\end{lemma}

The algorithm could be adapted to the rational setting but it will become heuristic as there is no such efficient formula in the rational case. 

\paragraph{Splitting quaternion algebras in characteristic 2}

In characteristic 2 (to the best of our knowledge) there is no direct algorithm for finding nontrivial zeros of ternary quadratic forms. However, there is a similar relation between split quaternion algebras and quadratic forms with nontrivial zeros. In \cite{ivanyos2018computing} the authors study the problem of finding primitive idempotents in full matrix algebras over $\mathbb{F}_q(t)$ given by a structure constant representation. In particular this encompasses the case of quaternion algebras over function fields of characteristic 2. The main idea of the algorithm is the following. One computes two maximal orders, one over $\mathbb{F}_q[t]$ and one over the ring of rational functions whose denominator has degree larger than the degree of the numerator (the maximal order corresponding to the degree valuation). This intersection can be computed using lattice reduction techniques. The intersection will be finite algebra over the base field $\mathbb{F}_q$ which contains a rank 1 element from the large algebra which can be retrieved by computing the structure of this algebra. This algorithm runs in polynomial-time but has not been implemented so far.


\section{Finding nontrivial zeros of quadratic forms over \texorpdfstring{$\mathbb{F}_{2^k}(t)$}{F_{2^k}(t)}}\label{sec:main}

In this section we concentrate on the case of characteristic $2$. In subsection \ref{qfqachar2} we recall some basic facts on quadratic forms and quaternion algebras in characteristic $2$. Then subsection \ref{localcomp} is devoted to developing local criteria for the existence of nontrivial zeros of quadratic forms using Hensel's lemma. In section \ref{zeroschar2} with the help of these criteria we propose an algorithm that decides whether or not a quadratic form is isotropic globally and if so the algorithm finds a nontrivial zero. We give an example in subsection \ref{example}. Finally, in section \ref{zerodivchar2} we apply the results in the previous section to finding a zero divisor in split quaternion algebras defined over quadratic extensions. Using the construction of quaternion algebras with given local splitting conditions this leads to finding, in polynomial time, zero divisors in degree two matrix rings over nonsplit quaternion algebras defined over the ground field.

\subsection{Quadratic forms and quaternion algebras in characteristic 2}\label{qfqachar2}

In this subsection we recall important facts about quadratic forms and quaternion algebras in characteristic 2. Our main source is \cite[Chapter 6]{voight2018quaternion}.
From here on $F$ will always denote a field with characteristic $2$.

\begin{lemma}\label{kvatalg}\cite[Chapter 6]{voight2018quaternion}
For every quaternion algebra $A$ over $F$ there exists an $F$-basis $1,i,j,k$ of $A$ such that
\[
i^2+i=a,\, j^2=b, \quad \text{and} \quad k=ij=j(i+1)
\]
where $a,b\in F$.
\end{lemma}

We denote the quaternion algebra over $F$ with parameters $a,b$ as $\left[\frac{a,b}{F}\right)$. We recall some facts about quadratic forms over fields of characteristic 2.


\begin{definition}
A \textbf{quadratic form} over $F$ is a homogeneous polynomial $Q$ of degree two in $n$ variables $x_1,\ldots,x_n$ for some $n$. We say that $Q$ is \textbf{isotropic} if there exist $a_1,\ldots,a_n\in F$ not all zero such that $Q(a_1,\ldots,a_n)=0$. If $Q$ is not isotropic, we say that $Q$ is \textbf{anisotropic}.
\end{definition}

We can also view a quadratic form $Q$ with $n$ variables over $F$ as a $Q: F^n\to F$ function. This motivates the following definition.
\begin{definition}
We say that two quadratic forms $Q_1$ and $Q_2$ are \textbf{isometric} if there exists a $\varphi: F^n\to F^n$ invertible linear map such that $Q_1\circ \varphi=Q_2$. 
\end{definition}

\begin{definition}
Let $Q_1$ and $Q_2$ be diagonal quadratic forms in $n$ variables. We call $Q_1$ and $Q_2$ \textbf{similar} if there exist a quadratic form $Q'$ that is isometric to $Q_2$ and such that $Q'$ can be obtained from $Q_1$ by multiplication of $Q_1$ by a non-zero $g\in F$.
\end{definition}

Even though if $\text{char}\, F=2$, not all quadratic forms can be diagonalized (we get $ax^2+axy+by^2$ as the general form), the following can be said about quadratic forms in four variables. 

\begin{lemma}\label{diagalak}\cite[Cor. 7.32]{elman2008algebraic} 
Every regular quadratic form in four variables over $F$ is equivalent to a quadratic form in the form of
\[
a_1x_1^2+x_1x_2+b_1x_2^2+a_3x_3^2+x_3x_4+b_2x_4^2
\]
where $a_1,a_3,b_1,b_2\in F$.
\end{lemma}
\begin{corollary}\label{diagalakcor}
Every regular quadratic form in four variables over $F$ is equivalent to a quadratic form in the form of
\[
a_1x_1^2+a_1x_1x_2+a_1a_2x_2^2+a_3x_3^2+a_3x_3x_4+a_3a_4x_4^2
\]
where $a_1,a_2,a_3,a_4\in F$.
\end{corollary}
\begin{proof}
We start from the canonical form described in Lemma \ref{diagalak}.After substituting $x_2\gets a_1x_2$ and $x_4\gets a_3x_4$, we get that $a_1x_1^2+a_1x_1x_2+a_1^2b_1x_2^2+a_3x_3^2+a_3x_3x_4+a_3^2b_2x_4^2$. After setting $a_2=a_1b_1$ and $a_4=a_3b_2$ we arrive to the form $a_1x_1^2+a_1x_1x_2+a_1a_2x_2^2+a_3x_3^2+a_3x_3x_4+a_3a_4x_4^2$.
\end{proof}

The following lemma \cite[Theorem 6.4.11]{voight2018quaternion} highlights a connection between the isotropy of quadratic forms and the splitting of quaternion algebras: 
\begin{lemma}[Hilbert equation]\label{hilberteq}
A quaternion algebra $\left[\frac{a,b}{F}\right)$ is split if and only if $bx^2+bxy+aby^2=1$ has a solution with $x,y\in F$.
\end{lemma}

Now if $X^2+X+a$ has a solution in $F$ (in this case put $K_a:=F$) then the form $x^2+xy+ay^2$ is equivalent to $x^2+xy$ which represents all elements in $F$. Otherwise let $\alpha$ be a root of the polynomial $X^2+X+a$ in a quadratic extension $K_a/F$. Then $x^2+xy+ay^2=N_{K/F}(x+y\alpha)$ is the norm form. Therefore in case $F$ is a local field of characteristic $2$, we may apply Thm.\ \ref{LCFTneeded} to deduce

\begin{lemma}\label{localrecip}
Assume $F$ is a local field of characteristic $2$. Then we have $F^\times/N_{K_a/F}(K_a^\times)\cong \operatorname{Gal}(K_a/F)$ is cyclic of order at most $2$. The subgroup $N_{K_a/F}(K_a^\times)\leq F^\times$ is uniquely determined by the extension $K_a/F$. In particular, the regular quadratic form
\[
a_1x_1^2+a_1x_1x_2+a_1a_2x_2^2+a_3x_3^2+a_3x_3x_4+a_3a_4x_4^2
\]
has no nontrivial solutions in $F$ if and only if $K_{a_2}=K_{a_4}$ is a quadratic extension of $F$ such that exactly one of $a_1$ and $a_3$ is represented by the form $x^2+xy+a_2y^2$. We have $K_{a_2}=K_{a_4}$ if and only if $x^2+x+a_2+a_4$ splits in $F$.
\end{lemma}
\begin{proof}
If $K_{a_2}\neq K_{a_4}$ then the intersection $a_1N_{K_{a_2}/F}(K_{a_2}^\times)\cap a_3N_{K_{a_4}/F}(K_{a_4}^\times)$ is a full coset of $N_{K_{a_2}/F}(K_{a_2}^\times)\cap N_{K_{a_4}/F}(K_{a_4}^\times)$ and is therefore nonempty. The last statement follows from Artin--Schreier theory.
\end{proof}

In order to handle quadratic forms, just like in odd characteristics, we will need to introduce a quadratic residue symbol. If $\F$ is a finite field of characteristic $2$ and $\pi$ is an irreducible polynomial in $\F[t]$, then every element in $\F[t]/(\pi)$ will be a square (as the factor ring is a finite field of characteristic 2), so the definition will need to differ slightly. The following definition and lemma with proof can be found in \cite{conrad2010quadratic}.

\begin{definition}
For a monic irreducible $\pi$ in $\F[t]$ and any $f\in \F(t)$ that has no pole at $\pi$, let
\[
[f,\pi):=
\begin{cases}
      0, & \text{if}\ f\equiv x^2+x\pmod{\pi} \ \text{for some} \ x\in \F[t] \\
      1, & \text{otherwise}
    \end{cases}
\]
If $[f,\pi)=0$, we say that $f$ is a \textbf{quadratic residue} modulo $\pi$. Similarly, for the place at $\infty$ we define
\[
[f,\infty):=
\begin{cases}
      0, & \text{if}\ f\equiv x^2+x \pmod{t^{-1}}\ \text{for some} \ x\in \F[t^{-1}] \\
      1, & \text{otherwise}
    \end{cases}
\]
whenever $f\in\F(t)$ has no pole at $\infty$ (ie.\ $\deg(f)\leq 0$). If $f$ has a pole at the (finite or infinite) place $\pi$ then $[f,\pi)$ has no meaning.
\end{definition} 

\begin{lemma}\label{kvadmaradekprop}
The symbol $[f,\pi)$ has the following properties:
\begin{enumerate}
    \item if $f_1\equiv f_2 \,\pmod{\pi} $, then $[f_1,\pi)=[f_2,\pi)$,
    \item $[f,\pi)\equiv f+f^2+\ldots+f^{q^{\text{deg} \pi}/2}\, \pmod{\pi}$, where $q=|F|$,
    \item $[f_1+f_2,\pi)=[f_1,\pi)+[f_2,\pi)$,
    \item $[f^2+f,\pi)=0$.
\end{enumerate}
\end{lemma}

\subsection{Local lemmas}\label{localcomp}

We denote by $v_f$ the $f$-adic valuation on $\FF_{2^k}(t)$ for a (finite or infinite) prime $f\in \FF_{2^k}(t)$, by $\FF_{2^k}(t)_{(f)}$ the $f$-adic completion, and by $\FF_{2^k}(t)_{(f)}^+:=\{u\in \FF_{2^k}(t)_{(f)}\mid v_f(u)\geq 0\}$ its valuation ring. 

We are interested in the range of the quadratic form $x^2+xy+ay^2$ for some $a\in \FF_{2^k}(t)$. Note that this is the norm form of the quadratic Artin--Schreier extension adjoining the root of $X^2+X+a$.

\begin{definition}\label{minminmin}
For $a \in \FF_{2^k}(t)$ we call the quadratic form $x^2+xy+ay^2$ \emph{minimal} if all the poles of $a$ (including $\infty$) have odd multiplicity. 
\end{definition}
Note that for a finite prime $f$ the multiplicity of the the pole of $a$ is by definition the exponent of $f$ in the denominator of $a$. The multiplicity of the pole of $a$ at $\infty$ is the degree of $a$ if it is positive and $0$ otherwise. The only elements $a\in \FF_{2^k}(t)$ that have no poles are the constants $a\in \FF_{2^k}$ in which case the form $x^2+xy+ay^2$ is minimal. The fact that each norm form is equivalent to a minimal follows easily from Artin--Schreier theory. We include an algorithmic proof as we need its running time.

\begin{lemma}\label{minimalform}
Any quadratic form $x^2+xy+ay^2$ with is equivalent to a minimal form. The equivalent minimal form can be found in polynomial time. 
\end{lemma}

\begin{proof}
Assume $a=\frac{g_1}{f^{2r}h_1}$ with $f\nmid g_1,h_1$ for some finite prime $f$. Since $\FF_{2^k}[t]/(f)$ is a finite field of characteristic $2$, the $2$-Frobenius is bijective on $\FF_{2^k}[t]/(f)$. In particular, there exists a polynomial $g\in\FF_{2^k}[t]$ such that $f\mid g^2h_1+g_1$ (one can find $g$ by squaring $\frac{g_1}{h_1}$ $k\deg(f)-1$-times modulo $f$). So we may replace the variable $x$ by $x_1=x+\frac{gy}{f^r}$ to obtain
\begin{align*}
x^2+xy+ay^2=x_1^2+\frac{g^2y^2}{f^{2r}}+x_1y+\frac{gy^2}{f^r}+\frac{g_1y^2}{f^{2r}h_1}=\\
=x_1^2+x_1y+\frac{g^2h_1+g_1+f^rh_1g}{f^{2r}h_1}y^2
\end{align*}
and $a':=\frac{g^2h_1+g_1+f^rh_1g}{f^{2r}h_1}$ has one less $f$ in the denominator. Repeating the process for all finite primes in the denominator of $a$ we are reduced to handle the case of the infinite prime. This is entirely analogous: assume we have $a=\frac{g_1}{h_1}$ with $2r:=\deg g_1-\deg h_1$ even and positive. Since the leading coefficient of $a$ is a square in $\mathbb{F}_{2^k}$, there exists $0\neq c\in \mathbb{F}_{2^k}$ such that $\deg (g_1+c^2t^{2r}h_1)<\deg g_1$. Therefore putting $x_1=x+ct^ry$ we obtain the form
\begin{align*}
x^2+xy+ay^2=x_1^2+c^2t^{2r}y^2+x_1y+ct^ry^2+\frac{g_1}{h_1}y^2=\\
=x_1^2+x_1y+\frac{h_1ct^r+h_1c^2t^{2r}+g_1}{h_1}y^2
\end{align*}
such that $a'=a+ct^r+c^2t^{2r}=\frac{h_1ct^r+h_1c^2t^{2r}+g_1}{h_1}$ has smaller degree than $a$. Repeating this step several times we deduce the statement.
\end{proof}
\begin{remark}
The above proof also shows that the minimal form of $x^2+xy+ay^2$ is unique up to an additive constant of the form $\alpha^2+\alpha$ with $\alpha\in\FF_{2^k}$.
\end{remark}

By the local-global principle (Theorem \ref{HM}) we are reduced to identifying the range of a minimal quadratic form $x^2+xy+ay^2$ locally at each place $f$ of $\FF_{2^k}(t)$. We may apply Lemma \ref{localrecip}, however, for our purposes we need to identify explicit congruence conditions on $c$ being in the range. Put $K_{a,f}$ for the splitting field of the polynomial $X^2+X+a$ over $\FF_{2^k}(t)_f$. We distinguish two cases whether or not $a$ has a pole at $f$, ie.\ whether or not the splitting field of $X^2+X+a$ ramifies at $f$. At first we treat the case when $a$ is an $f$-adic integer. 

\begin{lemma}\label{2valtmohatosag} Assume $v_f(a)\geq 0$.
\begin{enumerate}
    \item If $v_f(c)$ is even then the equation $x^2+xy+ay^2=c$ has a solution in $\FF_{2^k}(t)_{(f)}$.
    \item If $v_f(c)$ is odd then the equation $x^2+xy+ay^2=c$ has a solution in $\FF_{2^k}(t)_{(f)}$ if and only if $[a,f)=0$. 
\end{enumerate}
\end{lemma}

\begin{proof}
Note that by Hensel's lemma the extension $K_{a,f}/\FF_{2^k}(t)_f$ is unramified. Therefore by Thm.\ \ref{LCFTneeded} the image of the norm map $N_{K_{a,f}/\FF_{2^k}(t)_f}\colon K_{a,f}^\times\to \FF_{2^k}(t)_f^\times$ contains the group of units $(\FF_{2^k}(t)_f^+)^\times$ in the ring of integers. Further, $N_{K_{a,f}/\FF_{2^k}(t)_f}$ is onto if and only if $K_{a,f}=\FF_{2^k}(t)_f$. The latter is equivalent to $[a,f)=0$.
\end{proof}



\begin{lemma}\label{4valtmohatosag}
Let $a_1,a_3\in \FF_{2^k}[t]$ be square-free polynomials with $\gcd(a_1,a_3)=1$ and $a_2,a_4\in \FF_{2^k}(t)$. Let $f$ be a place, ie.\ either a monic irreducible polynomial or $f=\infty$ such that $v_f(a_1a_3)$ is odd. Assume that neither $a_2$ nor $a_4$ has a pole at $f$. Then the equation $a_1x_1^2+a_1x_1x_2+a_1a_2x_2^2+a_3x_3^2+a_3x_3x_4+a_3a_4x_4^2=0$ has a nontrivial solution in $\FF_{2^k}(t)_{(f)}$ if and only if at least one of the two conditions holds:
\begin{enumerate}
    \item $[a_2,f)=0$
    \item $[a_4,f)=0$
\end{enumerate}
\end{lemma}

\begin{proof}
This is a combination of Lemmas \ref{localrecip} and \ref{2valtmohatosag}.
\end{proof}

Now we turn our attention to the case when $a$ has a pole at $f$ (ie.\ $K_{a,f}$ ramifies). Note that unlike in the case of characteristic $0$ there exist infinitely many ramified quadratic extensions of local fields of characteristic $2$. By Lemma \ref{minimalform} the pole must be of odd degree $2r+1$ therefore the following lemma is relevant. In this case it is more convenient to multiply by $f^{2r+1}$ and put $b=af^{2r+1}$ which is an $f$-adic unit. Note that  $c$ is in the range of the quadratic form $x^2+xy+ay^2$ if and only if so is $cd^2$ for all $0\neq d\in \FF_{2^k}(t)$ therefore we may rescale $c$ by a square element as convenient.

\begin{lemma}\label{3valtsolv}
Let  $b,c$ be in $\FF_{2^k}(t)_{(f)}$ such that $v_f(b)=0$ (ie.\ $b$ is an $f$-adic unit) and $v_f(c)=0$ or $1$. Then the equation $$f^{2r+1}x^2+f^{2r+1}xy+by^2=cf^{2r}$$ has a solution in $\FF_{2^k}(t)_{(f)}$ if and only if it has a solution modulo $f^{4r+3}$. All such solutions lie in the valuation ring $\FF_{2^k}(t)_{(f)}^+$.
\end{lemma}
\begin{proof}
$\Rightarrow:$ Suppose we have a solution $(u,v)\in \FF_{2^k}(t)_{(f)}$. Assume for contradiction that one of $u$ and $v$ is not in $\FF_{2^k}(t)_{(f)}^+$. Multiplying by the square of the common denominator $f^l$ of $u$ and $v$ we obtain $u_1=f^lu,v_1=f^lv\in \FF_{2^k}(t)_{(f)}^+$ such that $f^{2r+2l}\mid f^{2r+1}u_1^2+f^{2r+1}u_1v_1+bv_1^2$ but $f$ does not divide at least one of $u_1$ and $v_1$. Since $f\nmid b$ we obtain $f^{2r+1}\mid v_1^2$ whence $f^{r+1}\mid v_1$. So we deduce $f^{2r+2}\mid f^{2r+1}u_1v_1+bv_1^2$ and $f^{2r+2}\mid f^{2r+1}u_1^2$ contradicting to $f\nmid u_1$. Hence we may reduce the equality $f^{2r+1}u^2+f^{2r+1}uv+bv^2=cf^{2r}$ modulo $f^{4r+3}$.

$\Leftarrow:$ Assume we have $u_0,v_0\in \FF_{2^k}(t)_{(f)}^+$ such that
$$c_0f^{2r}:=f^{2r+1}u_0^2+f^{2r+1}u_0v_0+bv_0^2\equiv cf^{2r}\pmod{f^{4r+2}}\ .$$ Then we must have $f^r\mid v_0$ and put $v_0=f^rv_1$ so dividing by $f^{2r}$ we deduce
$$c_0=fu_0^2+f^{r+1}u_0v_1+bv_1^2\equiv c\pmod{f^{2r+2}} $$
Since $f^2\nmid c$ at least one of $u_0$ and $v_1$ is not divisible by $f$. Putting $c_1:=\frac{c-c_0}{f^{2r+2}}$, we look for the solution of the original equation in the form $x=u_0+f^{r+1}x_1$, and $y=v_0+f^{2r+1}y_1$. So we are reduced to solving the equation

\begin{align*}
f^{2r+1}(u_0+f^{r+1}x_1)^2+f^{2r+1}(u_0+f^{r+1}x_1)(f^rv_1+f^{2r+1}y_1)+\\+b(f^rv_1+f^{2r+1}y_1)^2
=f^{2r}(c_0+f^{2r+2}c_1)\ .
\end{align*}
Using the equation for $c_0$ and dividing by $f^{4r+2}$ we obtain the equivalent equation
\begin{align}
fx_1^2+x_1v_1+u_0y_1+f^{r+1}x_1y_1+by_1^2=c_1\ .\label{modf3}
\end{align}
Now note that Hensel's lemma applies to \eqref{modf3} since the gradient 
$$
\left(\frac{\partial}{\partial x_1}(fx_1^2+v_0x_1+u_0y_1+fx_1y_1+by_1^2-c_1),\right.$$
$$\left. \frac{\partial}{\partial y_1}(fx_1^2+v_0x_1+u_0y_1+fx_1y_1+by_1^2-c_1)\right)=$$
$$=(v_1+f^{r+1}y_1,u_0+f^{r+1}x_1)\equiv (v_1,u_0)\pmod{f}
 $$

is nonzero modulo $f$. Therefore $(u_0,v_1)$ lifts to a solution modulo $f^{4r+3}$ $\Leftrightarrow$ \eqref{modf3} has a solution modulo $f$ $\overset{\text{Hensel}}{\Leftrightarrow}$ \eqref{modf3} has a solution in $\FF_{2^k}(t)_{(f)}$ $\Leftrightarrow$ $(u_0,v_0)$ lifts to a solution of $f^{2r+1}x^2+f^{2r+1}xy+by^2=cf^{2r}$ in $\FF_{2^k}(t)_{(f)}$.
\end{proof}

\subsection{Finding nontrivial zeros}\label{zeroschar2}

Let $Q(x_1,x_2,x_3,x_4)=a_1x_1^2+a_1x_1x_2+a_1a_2x_2^2+a_3x_3^2+a_3x_3x_4+a_3a_4x_4^2$ where $a_i\mathbb{F}_{2^k}(t)$. In this section we provide an algorithm for deciding whether $Q$ admits a nontrivial zero and if so, returns a nontrivial zero $(x_1,x_2,x_3,x_4)$. The main idea is similar to the main algorithm of \cite{ivanyos2019explicit}. We replace $Q$ with a similar form $Q'$ and then decide whether $Q'$ has a nontrivial zero using the local-global principle. If so, then we look for a common $c\in\mathbb{F}_{2^k}(t)$ which is represented by both $a_1x_1^2+a_1x_1x_2+a_1a_2x_2^2$ and $a_3x_3^2+a_3x_3x_4+a_3a_4x_4^2$ and then solve the equations $a_1x_1^2+a_1x_1x_2+a_1a_2x_2^2=c$ and $a_3x_3^2+a_3x_3x_4+a_3a_4x_4^2=c$ separately. Solving these equations is equivalent to finding zero divisors in quaternion algebras over $\mathbb{F}_{2^k}(t)$. This is a special case of the main algorithm from \cite[Section 4]{ivanyos2018computing}. 

\begin{theorem}\label{finding}
Let $Q(x_1,x_2,x_3,x_4)=a_1x_1^2+a_1x_1x_2+a_1a_2x_2^2+a_3x_3^2+a_3x_3x_4+a_3a_4x_4^2$ where $a_i\in\mathbb{F}_{2^k}(t)$. Then there exists a polynomial-time algorithm which decides whether $Q$ is isotropic and if so it finds a nontrivial zero of $Q$.
\end{theorem}
\begin{proof}
 By rescaling and dividing by common factors we can assume that $a_1,a_3\in\mathbb{F}_{2^k}[t]$ and $gcd(a_1,a_3)=1$. We look for a common $c\in\mathbb{F}_{2^k}[t]$ which is represented by both $a_1x_1^2+a_1x_1x_2+a_1a_2x_2^2$ and $a_3x_3^2+a_3x_3x_4+a_3a_4x_4^2$. Note that $c$ is represented by both these forms if and only if it is represented by both forms locally at each place $f$. By Lemma \ref{minimalform} we may assume that both $a_1x_1^2+a_1x_1x_2+a_1a_2x_2^2$ and $a_3x_3^2+a_3x_3x_4+a_3a_4x_4^2$ are minimal (in the sense of Definition \ref{minminmin}). Denote by $S$ the set of places where at least one of the following holds: 
\begin{enumerate}
\item $a_2$ has a pole at $f$;
\item $a_4$ has a pole at $f$;
\item $v_f(a_1a_3)$ is odd.
\end{enumerate}
We look for $c$ in the form $c=f_1f_2\cdots f_m h$ where $f_1,\dots,f_m\in S$ are monic irreducible polynomials and $h$ is irreducible. If $f\notin S$, then $v_f(a_1)$ and $v_f(a_3)$ have the same parity. Since $a_1$ and $a_3$ are coprime polynomials, their valuations must actually be even. If $v_f(c)=0$ the forms represent $c$ locally at $f$ by Lemma \ref{2valtmohatosag}(1). On the other hand, if $f\in S$ then we distinguish two cases.

First assume that neither $a_2$ nor $a_4$ has a pole at $f$ (whence $v_f(a_1a_3)$ is odd). Then whether or not a square-free polynomial $c$ is represented by the form $a_1x_1^2+a_1x_1x_2+a_1a_2x_2^2$ (resp.\ $a_3x_3^2+a_3x_3x_4+a_3a_4x_4^2$) depends only on the class of $c$ modulo $f^2$. So we may decide by checking all the residue classes modulo $f^2$ whether there is a common value $c$ of the two forms. If there is no common value then we are done (the $4$-variable form is not isotropic). By Lemma \ref{4valtmohatosag} this happens if and only if $[a_2,f)=[a_4,f)=1$. We put $f$ among $f_1,\dots,f_m$ if all the common square-free values of the two forms are divisible by $f$. Either way, there possibly appears a condition on $c$ modulo $f^2$ (which we shall encode in the choice of $h$).

Now assume that either $a_2$ or $a_4$ has a pole at $f$. We use Lemma \ref{localrecip} in order to decide whether there is a common value of the forms $a_1x_1^2+a_1x_1x_2+a_1a_2x_2^2$ and $a_3x_3^2+a_3x_3x_4+a_3a_4x_4^2$ locally at $f$: If $a_2+a_4$ has a pole of odd degree at $f$ then $x^2+x+a_2+a_4$ does not split in $\FF_{2^k}(t)_{(f)}$ hence there exist common values. We reduce the poles at $f$ in $a_2+a_4$ of even degree by adding elements of the form $\frac{d^2}{t^{2s}}+\frac{d}{t^s}$ ($d\in \FF_{2^k}[t]$). If all the poles are removed then we check whether $x^2+x+a_2+a_4$ splits by computing the symbol $[a_2+a_4,f)$ as in the previous case. Finally, for finding common values we use random values $c$ modulo $f^{4r+3}$ by Lemma \ref{3valtsolv} where $2r+1:=\max(-v_f(a_2),-v_f(a_4))$ is the bigger order of the pole at $f$ of $a_2$ and $a_4$. Again, if $v_f(c)$ is odd then we put $f$ into the finite set $\{f_1,\dots,f_m\}$.


Finally, if $f=\infty\in S$ then the congruence condition on $c$ involves a condition on the parity of the degree of $c$, as well as a condition modulo a power of $t$. Even if $\infty\notin S$ then the condition $v_\infty(c)=0$ means the degree of $c$ must be even.

Now if none of the above congruence conditions were contradictory then we deduce that the $4$-variable form is isotropic by Theorem \ref{HM}. So we proceed with finding a nontrivial zero looking for $c=f_1\cdots f_m h$ where the monic irreducible polynomials $f_1,\dots,f_m\in S$ are determined above and we choose $h$ irreducible satisfying all the above congruence conditions (including possibly a condition at $\infty$ if it belongs to $S$). This is possible by Lemma \ref{number}. By construction, $c$ is a common value of $a_1x_1^2+a_1x_1x_2+a_1a_2x_2^2$ and $a_3x_3^2+a_3x_3x_4+a_3a_4x_4^2$ locally at all places in $S$. Further, if $g\neq h$ is a (finite or infinite) place not in $S$ then $c$ is also a common value locally at $g$, so the only exception could be at $h$. However, by Hilbert's reciprocity law (Theorem \ref{Product formula}) $c$ is also a common value locally at $h$. 
\begin{algorithm}
	\KwIn{$a_1,a_3 \in \mathbb{F}_{2^k}[t]$ and $a_2,a_4 \in \mathbb{F}_{2^k}(t)^\times$ such that $a_1$ and $a_3$ are coprime, and every pole of $a_2$ (resp. $a_4$) has an odd order.}
	\KwOut{$h \in \mathbb{F}_{2^k}[t]$ such that $h$ is represented by both binary quadratic forms $Q_1(x,y) = a_1(x^2+xy+a_2y^2)$ and $Q_2(x,y) = a_3(x^2+xy+a_4y^2)$, or $\bot$ if no such $h$ exists.}
	$Conds \leftarrow [\ ]$\;
	$c \leftarrow 1$\;
	\For{$g \in \mathrm{Poles}(a_2) \cup \mathrm{Poles}(a_4)$}{
		\tcc{Apply lemmas \ref{3valtsolv} and \ref{2valtmohatosag} to find $h_g$}
		\If{$\exists h_g \in \mathbb{F}_{2^k}[t], N \in \mathbb{N} \mid \nu_g(h_g) \leq 1$ and  $\forall h \in \mathbb{F}_{2^k}(t)_{(g)}, h$ is represented by $Q_1$ and $Q_2$ if\\ $h = h_g \mod g^N$}{
			$\mathrm{Append}(\mathrm{Conditions},(h_g,g^N))$\;
			$c \leftarrow c \times g^{\nu_g(h_g) \mod 2}$\;
		}
		\Else{
			\Return{$\bot$}
		}
	}
	\For{$g \in \mathbb{F}_{2^k}(t)$ irreducible, such that $\nu_g(a_1 a_3)$ is odd and $g \notin \mathrm{Poles}(a_2) \cup \mathrm{Poles}(a_4)$}{
		\tcc{Apply lemmas \ref{2valtmohatosag} and \ref{4valtmohatosag} to find $\nu_g$}
		\If{$\exists \nu \in \{0,1\} \mid \forall h \in \mathbb{F}_{2^k}(t)_{(g)}, h$ is represented by $Q_1$ and $Q_2$ if $\nu_g(h) = \nu \mod 2$}{
			$c \leftarrow c \times g^\nu$\;
		}
		\Else{
			\Return{$\bot$}
		}
	}
	$i \leftarrow \mathbb{F}_{2^k}$-automorphism of $\mathbb{F}_{2^k}(t)$ sending $t$ to $1/X$\;
	$\nu_\infty \leftarrow \max(\deg(a_1),\deg(a_3))$\;
	\tcc{Use remark \ref{commonvalueinf} to find $h_\infty$ and  $N_\infty$}
	\If{$\exists h_\infty \in \mathbb{F}_{2^k}[t],N_\infty \in \mathbb{N} \mid \forall h \in \mathbb{F}_{2^k}(t)_{(t)}, h$ is represented by $Q_1$ and $Q_2$ if $\exists N \in \mathbb{N}, t^{2N + \nu_\infty}i(h) = h_\infty \mod t^{N_\infty}$}{
		\Return{$cg: g \in \mathbb{F}_{2^k}[X]$ is irreducible, $\deg{(cg)} = \nu_\infty + \deg{(h_\infty)} \mod 2$, $\forall (h,f) \in \mathrm{Conditions}, cg = h \mod f$ and $\exists n \in \mathbb{Z} \mid t^{2n+\nu_\infty} i(cg) = h_\infty \mod t^{N_\infty}$}

	}
	\Else{
		\Return{$\bot$}
	}
	\caption{SplitQuaternaryForm} \label{algofinding}
\end{algorithm}
\end{proof}
\begin{remark}\label{commonvalueinf}
	The case $f = \infty$ can be treated like the case $f = t$ after applying the automorphism fixing $\mathbb{F}_{2^k}$ and sending $t$ to $\frac{1}{t}$. It is necessary to multiply the equation by a power of $t$ to normalize it, and to check that if the new $a_2$ or $a_4$ have a pole at $t$, it has an odd order (otherwise one may apply again the algorithm from lemma \ref{minimalform}).
	When checking a candidate polynomial for this condition, one should be careful to normalize it with a power of $t$ of the same parity as the one used to derive the condition. The condition on the degree of $h$ then is the one that allows $h$ to have the prescribed valuation at the place $t$ after applying the automorphism and normalizing.
\end{remark}
\subsubsection{An example}\label{example}

We give a short example of how the algorithm works. Let $K=\mathbb{F}_2(t)$ and consider the form $a_1x_1^2+a_1x_1x_2+a_1a_2x_2^2+a_3x_3^2+a_3x_3x_4+a_3a_4x_4^2$ where $a_1=t^2+t+1,~a_2=t,~a_3=1,~a_4=1$. We have to look at $f$-adic solvability for $f=t^2+t+1$ and $f=\infty$ as these are the places for which the $f$-adic valuation of $a_1a_3$ is odd or $a_2$ has pole ($a_4$ is regular everywhere). 

One has $[a_2,t^2+t+1)=1$ and $[a_4,t^2+t+1)=0$ which implies that the $(t^2+t+1)$-adic valuation of a common value $c$ of $(t^2+t+1)(x_1^2+x_1x_2+tx_2^2)$ and $x_3^2+x_3x_4+x_4^2$ must be odd. Therefore we need to look for $c$ in form of $c=(t^2+t+1)h$ where $h$ is an irreducible polynomial over $\FF_2$. Further, $h=x_1^2+x_1x_2+tx_2^2$ admits a nonzero solution modulo $(t^2+t+1)$ if $h\equiv t\pmod{t^2+t+1}$. 
On the other hand, $v_\infty((t^2+t+1)h)=\deg((t^2+t+1)h)$ must be even in order for $x_3^2+x_3x_4+x_4^2=(t^2+t+1)h$ to be solvable at $\infty$ since the extension by a root of $X^2+X+1$ is unramified at $\infty$. Putting $z=1/t$ this boils down to the solvability of $zx_1^2+zx_1x_2+x_2^2=zh$ in $\FF_2[[z]]$ which is equivalent to $h\equiv 1\pmod{z^2}$ by Lemma \ref{3valtsolv}. After rescaling $h$ so that it is a polynomial in $t$ our condition is that the coefficient of $t^{2n-1}$ is zero in $h$ where $\deg(h)=2n$. 

Finally, a little computation shows that $h(t)=t^6+t+1$ satisfies 
\begin{enumerate}
\item $h\equiv t\pmod{t^2+t+1}$
\item $\deg(h)=2n$ is even
\item the coefficient of $t^{2n-1}$ in $h$ is zero
\end{enumerate}
So we are reduced to finding a solution to the following two equations globally:
\begin{enumerate}
    \item $x_1^2+x_1x_2+tx_2^2=(t^2+t+1)(t^6+t+1)$
    \item $x_3^2+x_3x_4+x_4^2=t^6+t+1$
\end{enumerate}
Running the known algorithms for binary forms we find $x_1=t^3+t^2+1$, $x_2=t$, $x_3=t^4+1$, $x_4=t^3$.

\subsection{Applications}\label{zerodivchar2}

In this subsection we give two applications of our results and methods.  One is to finding separable quadratic extensions $L$ of $\mathbb{F}_{2^k}(t)$ inside a quaternion algebra that is split by $L$. The other is constructing quaternion algebras over $\mathbb{F}_{2^k}(t)$ with prescribed Hasse invariants. 

\begin{theorem}\label{2quat}
Let $L$ be a separable quadratic extension of $\mathbb{F}_{2^k}(t)$ and let $B$ be a quaternion algebra over $\mathbb{F}_{2^k}(t)$ which is split by $L$. Then there exists a polynomial-time algorithm which finds a subfield of $B$ isomorphic to $L$. 
\end{theorem}
\begin{proof}
If $B$ is split, then one can find a explicit isomorphism between $B$ and $M_2(\mathbb{F}_{2^k}(t))$ in polynomial time using the main algorithm from \cite[Section 4]{ivanyos2018computing} (the algorithm also decides whether $B$ is split or not). From such an isomorphism a suitable maximal subfield can be constructed easily (by constructing a matrix whose minimal polynomial corresponds to a defining polynomial of $L$).

Now suppose that $B$ is a division algebra. In that case $B$ contains a maximal subfield isomorphic to $L$ \cite[Lemma 6.4.12]{voight2018quaternion}. Let $L=\mathbb{F}_{2^k}(t)(s)$ where $s^2+s=c$ and $c\in\mathbb{F}_{2^k}(t)$. If we find an element $u\in B$ such that $u^2+u=c$, then $u+s$ is a zero divisor as $u$ is not in the center. Suppose that $B$ has the following quaternion basis:
$$i^2+i=a$$
$$j^2=b$$
$$ij=j(i+1)$$
Let us look for $u$ in the form of $u=\lambda_1+\lambda_2 i+\lambda_3 j+\lambda_4 ij$, where $\lambda_i\in \FF_{2^k}(t)$.
\begin{align*}
u^2+u=\lambda_1^2+\lambda_2^2 a+\lambda_3^2 b+\lambda_4^2 ab+\lambda_3\lambda_4 b+\\
+\lambda_1+i(\lambda_2^2+\lambda_2)+j(\lambda_2\lambda_3+\lambda_3)+ij(\lambda_2\lambda_4+\lambda_4)
\end{align*}
For this to be in $\FF_{2^k}[t]$, $\lambda_2=1$ must hold. Now we investigate if the following equation has a non-trivial solution:
\begin{equation}\label{eq:2}
    \lambda_1^2+\lambda_3^2 b+\lambda_4^2 ab+\lambda_3\lambda_4 b+\lambda_1+a+c=0
\end{equation}
Let $\mu_2$ equal to the product of the denominators of all $\lambda_i$, $\mu_1:=\lambda_1\mu_2$, let us introduce new variables $\mu_3:=\lambda_3\mu_2$ and $\mu_4:=\lambda_4\mu_2$. Then multiplying (\ref{eq:2}) by $\mu_2^2$ gives
\begin{equation}\label{eq:3}
    \mu_1^2+\mu_1\mu_2+(a+c)\mu_2^2+b\mu_3^2 +b\mu_3\mu_4+ab\mu_4^2 =0
\end{equation}
where $\mu_1,\mu_2,\mu_3,\mu_4 \in \mathbb{F}_{2^k}[t]$. Now we find a solution to the above equation using the algorithm from Theorem \ref{finding} (\cite[Lemma 6.4.12]{voight2018quaternion} guarantees the existence of a solution) which returns $u$. The algorithm of Theorem \ref{finding} runs in polynomial time which implies the statement of the theorem.  
\end{proof}
\begin{remark}
The main motivation behind studying this algorithm is that in can be used to find zero divisors in split quaternion algebras over $L$. Namely, one frst constructs a subalgebra $B$ of the large algebra $A$ that is a quaternion algebra over $\mathbb{F}_{2^k}(t)$. If this algebra is not split, then it contains a subfield isomorphic to $L$ which is generated by some quaternion element $x$. Let $L=\mathbb{F}_{2^k}(t)(s)$. Then $x+s$ will be a zero divisor.
\end{remark}
The next proposition shows how to construct a quaternion division algebra with given Hasse invariants. 
\begin{proposition}\label{hasse}
Let $v_1,\dots,v_l$ be places of $\mathbb{F}_{2^k}(t)$ such that $l$ is even. Then there exists a polynomial-time algorithm which constructs a quaternion algebra over $\mathbb{F}_{2^k}(t)$ which is ramified exactly at $v_1,\dots,v_l$. 
\end{proposition}
\begin{proof}
Let $f_1,\dots f_m$ be the finite places amongst the $v_i$. First we find a monic irreducible polynomial in $b\in\mathbb{F}_{2^k}[t]$ such that $[b,f_i)=1$. This can be accomplished in the following way. One finds quadratic non-square $r_i$ modulo every $f_i$ ($\mathbb{F}_q[t]/(f_i)$ is finite field of cardinality $2^{\deg(f_i)k}$) and then obtains a residue class $r$ modulo $f_1\cdots f_m$ such that $r\equiv r_i\pmod{f_i}$ by Chinese remaindering. Then using Lemma \ref{number} one finds an irreducible polynomial of suitably large degree which is congruent to $r$ mod $f_1\cdots f_m$ by choosing random elements from the residue class until an irreducible is found. 

Let $a=f_1\cdots f_m$. We show that the quaternion algebra $A=[a,b)$ ramifies at every $f_i$. The algebra $A$ ramifies at $f_i$ if and only if the quadratic form $ax^2+axy+aby^2+z^2$ has a nontrivial zero in $\mathbb{F}_{2^k}(t)_{(f_i)}$. Since the form is homogeneous, it is enough to show that it does not admit an integral zero. The variable $z$ must be divisible by $f_i$ since $a$ is divisible by $f_i$. Now setting $z=f_iz'$ and dividing by $f_i$ we get the following equation:
$$a/f_ix^2+a/f_ixy+a/f_iby^2+f_iz'^2=0 $$
Suppose this equation has a nontrivial solution $(x_0,y_0,z_0)$. One may assume that $f_i$ does not divide $x_0,y_0$ and $z_0$ simultaneously. Then the following congruence condition holds:
$$a/f_ix_0^2+a/f_ix_0y_0+a/f_iby_0^2\equiv 0 \pmod{f_i} $$
Since $a/f_i$ is coprime to $f_i$ one can divide the congruence by $a/f_i$. If $y_0$ is not divisible by $f_i$, then $b$ is a quadratic residue mod $f_i$ which is a contradiction. If $y_0$ is divisible by $f_i$, then so is $x_0$. However, if $x_0$ and $y_0$ are both divisible by $f_i$, then $z_0$ is not divisible by $f_i$ and then $a/f_ix_0^2+a/f_ix_0y_0+a/f_iby_0^2+f_iz_0^2$ is not divisible by $f_i^2$ which is a contradiction. 

The algebra $A$ is split at $b$ since the equation $ax^2+axy+aby^2+z^2=0$ has a solution modulo $b$ (setting $z=0$ and $x=y=1$) which can be lifted by Hensel's lemma. $A$ is clearly split at all the other finite places and has the required splitting condition at $\infty$ by Hilbert reciprocity (Theorem \ref{Product formula}). 
\end{proof}
\begin{corollary}\label{divalg}
Let $D$ be a quaternion division algebra over $\mathbb{F}_{2^k}(t)$ and let $A$ be an algebra isomorphic to $M_2(D)$ given by structure constants. Then one can find a zero divisor in $A$ in polynomial time. 
\end{corollary}
\begin{proof}
We compute the local indices of $A$ using the algorithm \cite[Proposition 6.5.3.]{ivanyos1996algorithms} and then use Proposition \ref{hasse} to compute a division quaternion algebra $D_0$ with those exact invariants. Since we have constructed a structure constant representation of $D_0$, we can construct a structure constant representation of $M_2(D_0)$ by considering the basis where the matrix has one nonzero entry and that runs through the basis of $D_0$. Then as stated previously, one can construct an explicit isomorphism between $A$ and $M_2(D_0)$ from an explicit isomorphism between $A^{op}\otimes M_2(D_0)$ and $M_{16}(\mathbb{F}_{2^k}(t))$ in polynomial time using the main algorithm from \cite[Section 4]{ivanyos2018computing}. Finally, the preimage of the matrix 
$\begin{pmatrix}
1&0 \\
0&0
\end{pmatrix} $
is a zero divisor.
\end{proof}

\section{Implementation}\label{sec:impl}
    In this section, we give details about our implementation\footnote{\url{https://github.com/Char2QuadForms/Char2QuadForms}} of the algorithm in the Magma language \cite{magma}. We then provide details on the practical efficiency of our implementation and discuss the computational bottlenecks.
\subsection{Implementation details}\label{impldetails}
    The core of our code is a practical implementation of algorithm \ref{algofinding}. Our first step is to take as an input a quaternary quadratic form as a degree $4$ square matrix with coefficients in $\mathbb{F}_{2^k}(t)$. If $Q$ is the input quadratic form, we apply successively the reductions from lemma \ref{diagalak}, corollary \ref{diagalakcor} and lemma \ref{minimalform} to obtain coefficients $a_1,a_2,a_3,a_4 \in \mathbb{F}_{2^k}(t)$ which follow the hypotheses of theorem \ref{finding}, and such that the quadratic form $a_1(x_1^2 + x_1 x_2 + a_2 x_2^2) + a_3(x_3^2 + x_3 x_4 + a_4 x_4^2)$ is similar to $Q$. In addition, we make coefficients $a_1$ and $a_3$ square-free, as it simplifies computation and does not affect the place of the zeros.

    Our implementation of algorithm \ref{algofinding} follows the structure of the pseudo-code representation. See appendix \ref{appendix:algo} for more details on the subroutines for each case. Once all the conditions for a common value have been established, we randomly generate polynomials that satisfy said conditions until we find a prime polynomial. Because of the bound given in lemma \ref{number}, finding one such a polynomial can be done in probabilistic polynomial time. Once we find a prime polynomial which satisfies every condition, we move to the last step of the implementation.
    
    We independently solve equations $a_1(x_1^2+x_1 x_2+ a_2 x_2^2) = ch$ and $a_3(x_1^2+x_1 x_2+ a_3 x_2^2) = ch$. Both these equations directly reduce to a Hilbert equation, and then by lemma \ref{hilberteq} the problem reduces to finding an explicit isomorphism between a given quaternion algebra and the degree $2$ matrix algebra. In practice, we solve the equation $a_1(x_1^2+x_1 x_2 + a_2 x_2^2) = c$ by computing an explicit isomorphism between  $A = \left[\frac{a_2,a_1/c}{\mathbb{F}_{2^k}(t)}\right)$ and $M_2\left(\mathbb{F}_{2^k}(t)\right)$. It follows from \cite[equation 6.4.5]{voight2018quaternion} that in $A$, $nrd(x + yj + zk) = x^2 + \frac{a_1}{c}(y^2 + yz + a_2 z^2)$. We therefore find a singular matrix which pulls back to a quaternion of the form $x + yj + zk$ with $x \neq 0$, and then we set $x_1 = \frac{y}{x}$ and $x_2 = \frac{z}{x}$.
    
    We use the main algorithm from \cite{ivanyos2018computing} to compute the explicit isomorphism between $A$ and $M_2\left(\mathbb{F}_{2^k}(t)\right)$. Since, to the best of our knowledge, it has not been implemented yet, we provide an implementation in Magma which may be of independent interest.

\subsection{Computational data}
    
    In table \ref{table:data} we show the running time for some executions of algorithm \ref{algofinding}. This running time does not include solving the resulting ternary forms, which we will discuss separately. The tests were executed on the online Magma calculator\footnote{\url{http://magma.maths.usyd.edu.au/magma/}} with randomly generated polynomials. The degree of the input polynomials were not randomly chosen, but they were affected by the steps of minimization of the coefficients. We give the degrees of the coefficients after the minimisation steps. By the degree of a rational function we mean the maximal of the degrees of its numerator and of its denominator. The column $\deg h$ refers to the degree of the value represented by both binary forms, that is the output of algorithm \ref{algofinding}. The column $q$ refers to the cardinal of the finite field underlying our rational function field. The running times are given in seconds. We note that all running times here are given for input corresponding to an isotropic quadratic form. In general, it is faster for the algorithm to recognize an anisotropic form than to split an isotropic one.
    
    \begin{table}
    \centering
    \begin{tabular}{ccccccc}
    $\deg a_1$ & $\deg a_2$ & $\deg a_3$ & $\deg a_4$ & $\deg h$ & $q$ & Running time \\
    $83$ & $71$ & $7$ & $93$ & $384$ & $2$ & $0.570$ \\
    $204$ & $211$ & $1048$ & $211$ & $604$ & $2$ & $4.680$ \\
    $21$ & $25$ & $121$ & $25$ & $75$ & $2^{10}$ & $6.090$ \\
    $15$ & $21$ & $102$ & $21$ & $64$ & $2^{20}$ & $29.520$
    \end{tabular}
    \caption{Running times of algorithm \ref{algofinding}}
    \label{table:data}
    \end{table}

    The last part of our implementation is the main algorithm from \cite{ivanyos2018computing} (see the discussion in subsection \ref{impldetails}). A subroutine for this algorithm is the computation of a maximal order in a quaternion algebra. Since this subroutine was not implemented in Magma for algebras over fields of characteristic 2, we gave our own implementation using the polynomial time algorithm given in \cite[subsection 3.2]{ivanyos2018computing}. Our implementation of this algorithm runs significantly slower than the Magma built-in function for maximal order computation in odd characteristic. As a result, we do not draw conclusions regarding the running time for this part of the implementation.
    
    However, the implementation is still practical for small input. In table \ref{table:ternsolver} we give running time for our function solving equations of the form $x^2+xy+ay^2 = h$. Every line refers to a computation done over $\mathbb{F}_2(t)$.

    \begin{table}
    \centering
        \begin{tabular}{ccc}
            $\deg a$ & $\deg h$ & Running time \\
            $1$ & $4$ & $2.390$ \\
            $1$ & $8$ & $20.360$ \\
            $5$ & $14$ & $311.460$ \\
        \end{tabular}
        \caption{Running time for solving $x^2+xy+ay^2=h$}
        \label{table:ternsolver}
    \end{table}
\bibliographystyle{ACM-Reference-Format}
\bibliography{bib}


\begin{thebibliography}{18}


\ifx \showCODEN    \undefined \def \showCODEN     #1{\unskip}     \fi
\ifx \showDOI      \undefined \def \showDOI       #1{#1}\fi
\ifx \showISBNx    \undefined \def \showISBNx     #1{\unskip}     \fi
\ifx \showISBNxiii \undefined \def \showISBNxiii  #1{\unskip}     \fi
\ifx \showISSN     \undefined \def \showISSN      #1{\unskip}     \fi
\ifx \showLCCN     \undefined \def \showLCCN      #1{\unskip}     \fi
\ifx \shownote     \undefined \def \shownote      #1{#1}          \fi
\ifx \showarticletitle \undefined \def \showarticletitle #1{#1}   \fi
\ifx \showURL      \undefined \def \showURL       {\relax}        \fi
\providecommand\bibfield[2]{#2}
\providecommand\bibinfo[2]{#2}
\providecommand\natexlab[1]{#1}
\providecommand\showeprint[2][]{arXiv:#2}

\bibitem[\protect\citeauthoryear{Bosma, Cannon, and Playoust}{Bosma
  et~al\mbox{.}}{1997}]%
        {magma}
\bibfield{author}{\bibinfo{person}{Wieb Bosma}, \bibinfo{person}{John Cannon},
  {and} \bibinfo{person}{Catherine Playoust}.} \bibinfo{year}{1997}\natexlab{}.
\newblock \showarticletitle{The {M}agma algebra system. {I}. {T}he user
  language}.
\newblock \bibinfo{journal}{\emph{J. Symbolic Comput.}} \bibinfo{volume}{24},
  \bibinfo{number}{3-4} (\bibinfo{year}{1997}), \bibinfo{pages}{235--265}.
\newblock
\showISSN{0747-7171}
\urldef\tempurl%
\url{https://doi.org/10.1006/jsco.1996.0125}
\showDOI{\tempurl}
\newblock
\shownote{Computational algebra and number theory (London, 1993).}


\bibitem[\protect\citeauthoryear{Castel}{Castel}{2013}]%
        {castel2013solving}
\bibfield{author}{\bibinfo{person}{Pierre Castel}.}
  \bibinfo{year}{2013}\natexlab{}.
\newblock \showarticletitle{Solving quadratic equations in dimension 5 or more
  without factoring}.
\newblock \bibinfo{journal}{\emph{The Open Book Series}} \bibinfo{volume}{1},
  \bibinfo{number}{1} (\bibinfo{year}{2013}), \bibinfo{pages}{213--233}.
\newblock


\bibitem[\protect\citeauthoryear{Conrad}{Conrad}{2010}]%
        {conrad2010quadratic}
\bibfield{author}{\bibinfo{person}{Keith Conrad}.}
  \bibinfo{year}{2010}\natexlab{}.
\newblock \showarticletitle{Quadratic reciprocity in characteristic 2}.
\newblock \bibinfo{journal}{\emph{Unpublished notes, available at
  \url{http://www.math.uconn.edu/~kconrad/blurbs/ugradnumthy/QRchar2.pdf}}}
  (\bibinfo{year}{2010}).
\newblock


\bibitem[\protect\citeauthoryear{Cremona and Rusin}{Cremona and Rusin}{2003}]%
        {cremona2003efficient}
\bibfield{author}{\bibinfo{person}{John Cremona} {and} \bibinfo{person}{David
  Rusin}.} \bibinfo{year}{2003}\natexlab{}.
\newblock \showarticletitle{Efficient solution of rational conics}.
\newblock \bibinfo{journal}{\emph{Math. Comp.}} \bibinfo{volume}{72},
  \bibinfo{number}{243} (\bibinfo{year}{2003}), \bibinfo{pages}{1417--1441}.
\newblock


\bibitem[\protect\citeauthoryear{Cremona and van Hoeij}{Cremona and van
  Hoeij}{2006}]%
        {van2006solving}
\bibfield{author}{\bibinfo{person}{John Cremona} {and} \bibinfo{person}{Mark
  van Hoeij}.} \bibinfo{year}{2006}\natexlab{}.
\newblock \showarticletitle{Solving conics over function fields}.
\newblock \bibinfo{journal}{\emph{Journal de th{\'e}orie des nombres de
  Bordeaux}} \bibinfo{volume}{18}, \bibinfo{number}{3} (\bibinfo{year}{2006}),
  \bibinfo{pages}{595--606}.
\newblock


\bibitem[\protect\citeauthoryear{Elman, Karpenko, and Merkurjev}{Elman
  et~al\mbox{.}}{2008}]%
        {elman2008algebraic}
\bibfield{author}{\bibinfo{person}{Richard~S Elman}, \bibinfo{person}{Nikita
  Karpenko}, {and} \bibinfo{person}{Alexander Merkurjev}.}
  \bibinfo{year}{2008}\natexlab{}.
\newblock \bibinfo{booktitle}{\emph{The algebraic and geometric theory of
  quadratic forms}}. Vol.~\bibinfo{volume}{56}.
\newblock \bibinfo{publisher}{American Mathematical Soc.}
\newblock


\bibitem[\protect\citeauthoryear{Ivanyos}{Ivanyos}{1996}]%
        {ivanyos1996algorithms}
\bibfield{author}{\bibinfo{person}{G{\'a}bor Ivanyos}.}
  \bibinfo{year}{1996}\natexlab{}.
\newblock \emph{\bibinfo{title}{Algorithms for algebras over global fields}}.
\newblock \bibinfo{thesistype}{Ph.D. Dissertation}. \bibinfo{school}{Hungarian
  Academy of Sciences}.
\newblock


\bibitem[\protect\citeauthoryear{Ivanyos, Kutas, and R{\'o}nyai}{Ivanyos
  et~al\mbox{.}}{2018}]%
        {ivanyos2018computing}
\bibfield{author}{\bibinfo{person}{G{\'a}bor Ivanyos},
  \bibinfo{person}{P{\'e}ter Kutas}, {and} \bibinfo{person}{Lajos R{\'o}nyai}.}
  \bibinfo{year}{2018}\natexlab{}.
\newblock \showarticletitle{Computing Explicit Isomorphisms with Full Matrix
  Algebras over $\mathbb{F}_q(x)$}.
\newblock \bibinfo{journal}{\emph{Foundations of Computational Mathematics}}
  \bibinfo{volume}{18}, \bibinfo{number}{2} (\bibinfo{year}{2018}),
  \bibinfo{pages}{381--397}.
\newblock


\bibitem[\protect\citeauthoryear{Ivanyos, Kutas, and R{\'o}nyai}{Ivanyos
  et~al\mbox{.}}{2019}]%
        {ivanyos2019explicit}
\bibfield{author}{\bibinfo{person}{G{\'a}bor Ivanyos},
  \bibinfo{person}{P{\'e}ter Kutas}, {and} \bibinfo{person}{Lajos R{\'o}nyai}.}
  \bibinfo{year}{2019}\natexlab{}.
\newblock \showarticletitle{Explicit equivalence of quadratic forms over
  $\mathbb{F}_q(t)$}.
\newblock \bibinfo{journal}{\emph{Finite Fields and Their Applications}}
  \bibinfo{volume}{55} (\bibinfo{year}{2019}), \bibinfo{pages}{33--63}.
\newblock


\bibitem[\protect\citeauthoryear{Ivanyos and Sz{\'a}nt{\'o}}{Ivanyos and
  Sz{\'a}nt{\'o}}{1996}]%
        {ivanyos1996lattice}
\bibfield{author}{\bibinfo{person}{G{\'a}bor Ivanyos} {and}
  \bibinfo{person}{{\'A}gnes Sz{\'a}nt{\'o}}.} \bibinfo{year}{1996}\natexlab{}.
\newblock \showarticletitle{Lattice basis reduction for indefinite forms and an
  application}.
\newblock \bibinfo{journal}{\emph{Discrete Mathematics}} \bibinfo{volume}{153},
  \bibinfo{number}{1-3} (\bibinfo{year}{1996}), \bibinfo{pages}{177--188}.
\newblock


\bibitem[\protect\citeauthoryear{Koprowski}{Koprowski}{2021}]%
        {koprowski2021isotropic}
\bibfield{author}{\bibinfo{person}{Przemys{\l}aw Koprowski}.}
  \bibinfo{year}{2021}\natexlab{}.
\newblock \showarticletitle{Isotropic vectors over global fields}.
\newblock \bibinfo{journal}{\emph{arXiv preprint arXiv:2111.08569}}
  (\bibinfo{year}{2021}).
\newblock


\bibitem[\protect\citeauthoryear{Lam}{Lam}{2005}]%
        {lam2005introduction}
\bibfield{author}{\bibinfo{person}{Tsit-Yuen Lam}.}
  \bibinfo{year}{2005}\natexlab{}.
\newblock \bibinfo{booktitle}{\emph{Introduction to quadratic forms over
  fields}}. Vol.~\bibinfo{volume}{67}.
\newblock \bibinfo{publisher}{American Mathematical Soc.}
\newblock


\bibitem[\protect\citeauthoryear{Milne}{Milne}{1997}]%
        {milne1997class}
\bibfield{author}{\bibinfo{person}{James~S Milne}.}
  \bibinfo{year}{1997}\natexlab{}.
\newblock \showarticletitle{Class field theory}.
\newblock \bibinfo{journal}{\emph{lecture notes available at http://www. math.
  lsa. umich. edu/jmilne}} (\bibinfo{year}{1997}).
\newblock


\bibitem[\protect\citeauthoryear{Rauter}{Rauter}{1926}]%
        {rauter1926uber}
\bibfield{author}{\bibinfo{person}{H Rauter}.} \bibinfo{year}{1926}\natexlab{}.
\newblock \emph{\bibinfo{title}{{\"u}ber die Darstellbarkeit durch quadratische
  Formen im K{\"o}rper der rationalen Funktionen einer Unbestimmten {\"u}ber
  dem Restklassenk{\"o}rper mod p}}.
\newblock \bibinfo{thesistype}{Ph.D. Dissertation}. \bibinfo{school}{Halle}.
\newblock


\bibitem[\protect\citeauthoryear{Simon}{Simon}{2005a}]%
        {simon2005quadratic}
\bibfield{author}{\bibinfo{person}{Denis Simon}.}
  \bibinfo{year}{2005}\natexlab{a}.
\newblock \showarticletitle{Quadratic equations in dimensions 4, 5 and more}.
\newblock \bibinfo{journal}{\emph{preprint}} (\bibinfo{year}{2005}).
\newblock


\bibitem[\protect\citeauthoryear{Simon}{Simon}{2005b}]%
        {simon2005solving}
\bibfield{author}{\bibinfo{person}{Denis Simon}.}
  \bibinfo{year}{2005}\natexlab{b}.
\newblock \showarticletitle{Solving quadratic equations using reduced
  unimodular quadratic forms}.
\newblock \bibinfo{journal}{\emph{Math. Comp.}} \bibinfo{volume}{74},
  \bibinfo{number}{251} (\bibinfo{year}{2005}), \bibinfo{pages}{1531--1543}.
\newblock


\bibitem[\protect\citeauthoryear{Voight}{Voight}{2021}]%
        {voight2018quaternion}
\bibfield{author}{\bibinfo{person}{John Voight}.}
  \bibinfo{year}{2021}\natexlab{}.
\newblock \bibinfo{booktitle}{\emph{Quaternion algebras}}.
  \bibinfo{series}{Graduate texts in mathematics}, Vol.~\bibinfo{volume}{288}.
\newblock \bibinfo{publisher}{Springer,Cham}.
\newblock


\bibitem[\protect\citeauthoryear{Wan}{Wan}{1997}]%
        {wan1997generators}
\bibfield{author}{\bibinfo{person}{Daqing Wan}.}
  \bibinfo{year}{1997}\natexlab{}.
\newblock \showarticletitle{Generators and irreducible polynomials over finite
  fields}.
\newblock \bibinfo{journal}{\emph{Math. Comp.}} \bibinfo{volume}{66},
  \bibinfo{number}{219} (\bibinfo{year}{1997}), \bibinfo{pages}{1195--1212}.
\newblock


\end{thebibliography}
\newpage
\appendix
\section{Algorithm subroutines}\label{appendix:algo}

\begin{algorithm}
	\KwIn{$a_1,a_3,f \in \mathbb{F}_{2^k}[t], a_2,a_4 \in \mathbb{F}_{2^k}(t)$ such that $a_1$ and $a_3$ are nonzero coprime polynomials, each pole of $a_2$ and $a_4$ has an odd multiplicity, and $f$ is an irreducible polynomial which is a pole of $a_2$ or $a_4$.}
	\KwOut{$h_f \in \mathbb{F}_{2^k}[t],n \in \mathbb{N}$ such that for all $h \in \mathbb{F}_{2^k}[t]$, if $h = h_f \mod f^n$, then $h$ is represented by both binary forms $a_1(x^2+xy+a_2y^2)$ and $a_3(x^2+xy+a_4y^2)$ over $\mathbb{F}_{2^k}(t)_{(f)}$. Outputs $\bot$ instead if such a tuple does not exist.}
	$N \leftarrow 2\max(-\nu_f(a_2),-\nu_f(a_4))+1$\;
	$K_2 \leftarrow \mathbb{F}_{2^k}(t)/(t^2+t+a_2)$\;
	$K_4 \leftarrow \mathbb{F}_{2^k}(t)/(t^2+t+a_4)$\;
	$K_6 \leftarrow \mathbb{F}_{2^k}(t)/(t^2+t+a_2+a_4)$\;
	\If{$(\mathrm{not}\ K_6 = \mathbb{F}_{2^k}(t))\ \mathrm{or}\ \frac{a_1}{a_3} \in N_{K_2/\mathbb{F}_{2^k}(t)}(K_2^\times)$}{
		\Repeat{$\nu_f(c) \le 1\ \mathrm{and}\ c \in N_{K_2/\mathbb{F}_{2^k}(t)}(K_2^\times) \cap N_{K_4/\mathbb{F}_{2^k}(t)}(K_4^\times)$}{
			$c \xleftarrow{\mathrm{Random}} \mathbb{F}_{2^k}[t]/(f^N)$\;
		}
		\Return{c,N}
	}
	\Else{
		\Return{$\mathrm{\bot}$}
	}
	\caption{CommonValuePole}
\end{algorithm}

\begin{algorithm}
	\KwIn{$a_1,a_3,f \in \mathbb{F}_{2^k}[t], a_2,a_4 \in \mathbb{F}_{2^k}(t)$ such that $a_1$ and $a_2$ are nonzero coprime polynomials, each pole of $a_2$ and $a_4$ has an odd multiplicity, and $f$ is an irreducible polynomial which is not a pole of $a_2$ or $a_4$ but such that $\nu_f(a_1a_3)$ is odd.}
	\KwOut{$\nu \in \{0,1\}$ such that a polynomial $h \in \mathbb{F}_{2^k}[t]$ is represented by both binary quadratic forms $a_1(x^2 + x y + a_2 y^2)$ and $a_3(x^2 + xy + a_4 y^2)$ over $\mathbb{F}_{2^k}(t)_{(f)}$ if $\nu_f(h) = \nu \mod 2$. Outputs $\bot$ if there is no such $\nu$.}
	\If{$[a_1,f) = [a_3,f) = 1$}{
		\Return{$\bot$}
	}
	\ElseIf{$[a_1,f) = \nu_f(a_1) \mod 2$}{
		\Return{$1$}
	}
	\Else{
		\Return{$0$}
	}
\caption{CommonValueOdd} 
\end{algorithm}

\begin{algorithm}[H]
	\KwIn{$a_1,a_3 \in \mathbb{F}_{2^k}[t], a_2,a_4 \in \mathbb{F}_{2^k}(t)$ such that $a_1$ and $a_3$ are nonzero coprime polynomials, and each pole of $a_2$ and $a_4$ has an odd multiplicity.}
	\KwOut{$h_\infty \in \mathbb{F}_{2^k}[t],N_\infty \in \mathbb{N}$ such that a polynomial $h \in \mathbb{F}_{2^k}[t]$ is represented by both binary forms $a_1(x^2 + x y + a_2 y^2)$ and $a_3(x^2 + xy + a_4 y^2)$ over the completion of $\mathbb{F}_{2^k}(t)$ at the infinite place if $\deg(h) = \nu_t(h_\infty) \mod 2$ and $\exists n \in \mathbb{Z} \mid x^{2n+\nu_\infty} i(cg) = h_\infty \mod N_\infty$, with $\nu_\infty = \max{(\deg{a_1},\deg{a3})}$.}
$i \leftarrow \mathbb{F}_{2^k}$-automorphism of $\mathbb{F}_{2^k}(t)$ sending $t$ to $1/t$\;
	$\nu_\infty \leftarrow -\max_{i \in \{1,3\}}\deg(a_i)$\;
	$ia_1 \leftarrow i(x^{\nu_\infty} a_1), ia_3 \leftarrow i(x^{\nu_\infty} a_3)$\;
	$ia_2 \leftarrow i(a_2), ia_4 \leftarrow i(a_4)$\;
	\tcc{By minimalise we mean to apply the algorithm described in the proof of lemma \ref{minimalform}}
	minimalise(ia2)\;
	minimalise(ia4)\;
	\If{$t \in \mathrm{Poles}(ia_2) \cup \mathrm{Poles}(ia_4)$}{
		$\mathrm{Res} \leftarrow \mathrm{CommonValuePole}(ia_1,ia_2,ia_3,ia_4,t)$\;
		\If{$\mathrm{Res} = \bot$}{
			\Return{$\bot$}
		}
		\Else{
			$(h_\infty,N_\infty) \leftarrow Res$\;
		}
	}
	\ElseIf{$\nu_t(ia_1 ia_3) = 1 \mod 2$}{
		$\mathrm{Res} \leftarrow \mathrm{CommonValueOdds}(a_1,a_2,a_3,a_4,f)$\;
		\If{$\mathrm{Res} = \bot$}{
			\Return{$\bot$}
		}
		\Else{
			$h_\infty \leftarrow t^{\mathrm{Res}}$\;
			$N_\infty \leftarrow 0$\;
		}
	}
	\Else{
		$h_\infty \leftarrow 1$\;
		$N_\infty \leftarrow 0$\;
	}
	\Return{$(h_\infty,N_\infty)$}

	\caption{CommonValueInf}
\end{algorithm}




\end{document}